\documentclass[11pt,reqno]{amsart} 

\usepackage{epsfig}
\usepackage[utf8]{inputenc}
\usepackage{multimedia}
\usepackage{graphicx,color,blindtext}
\usepackage{caption}
\usepackage{subcaption}
\usepackage{mathtools}
\usepackage{ulem}
\usepackage[T1]{fontenc}
\usepackage{amssymb,amsmath}

\newtheorem{theorem}{Theorem}[section]
\newtheorem{proposition}{Proposition}[section]
\newtheorem{lemma}{Lemma}[section]
\newtheorem{remark}{Remark}[section]
\newtheorem{corollary}{Corollary}[section]
\newtheorem{definition}{Definition}[section]
\newtheorem{aplemma}{Lemma}

\newenvironment{proofth}[1]{\paragraph{\textit{#1}}}{\hfill$\square$}

\newcommand{\eps}{\varepsilon}

\newcommand{\C}{\mathcal{C}}
\newcommand{\A}{\mathcal{A}}

\newcommand{\D}{\mathcal{D}}

\newcommand\CC{\hbox{C\kern -.58em {\raise .54ex \hbox
			{$\scriptscriptstyle |$}}
		\kern-.55em {\raise .53ex \hbox{$\scriptscriptstyle |$}} }}
\newcommand\qd{\hfill$\sqcap\kern-8.0pt\hbox{$\sqcup$}$}
\newcommand\NN{\hbox{I\kern-.2em\hbox{N}}}
\newcommand\nn{\hbox{I\kern-.2em\hbox{N}}}
\newcommand\RR{I\!\!R}
\newcommand\sRR{{\sl \hbox{I\kern-.2em\hbox{R}}}}
\newcommand\QQ{\hbox{I\kern-.53em\hbox{Q}}}

\newcommand\sign{\hbox{Sign}}
\newcommand\signp{{\hbox{Sign}^+}}

\newcommand\heps{{\mathcal H}_\varepsilon}

\newcommand\R{\mathbb R}
\newcommand\N{\mathbb N}

\everymath{\displaystyle}

\numberwithin{equation}{section}

\title[A granular model for crowd motion and pedestrian flow]{A granular model for crowd motion\\ and pedestrian flow}

\author[N. Igbida]{Noureddine Igbida}
\address{Institut de recherche XLIM, UMR-CNRS 7252, Facult\'e des Sciences et Techniques, Universit\'e de Limoges, 87100 Limoges,  France.}{} 
\email{noureddine.igbida@unilim.fr} 

\author[J.M.~Urbano]{Jos\'{e} Miguel Urbano}
\address{Applied Mathematics and Computational Sciences (AMCS), Computer, Electrical and Mathematical Sciences and Engineering Division (CEMSE), King Abdullah University of Science and Technology (KAUST), Thuwal, 23955 -6900, Kingdom of Saudi Arabia and CMUC, Department of Mathematics, University of Coimbra, 3000-143 Coimbra, Portugal}{} 
\email{miguel.urbano@kaust.edu.sa}

\begin{document}
	
	
	
	
	
	
	\begin{abstract}
		We study a granular model for congested crowd motion and pedestrian flow. Our approach is based on an approximation through a Hele-Shaw type equation involving a degenerate operator of $p-$Laplacian type and a linear drift, for which we prove existence and uniqueness using nonlinear semigroup methods and the doubling variables technique. Our main result shows that, as $p \to \infty$, the weak solutions of the $p-$problem converge to a solution of the congested crowd motion problem interpreted in a variational sense.  
	\end{abstract}
	
	\date{\today}
	
	\maketitle
	
	\section{Introduction}
	
	Macroscopic models for pedestrian flow, in which the crowd is modeled as a moving fluid, were first introduced in \cite{Bord} and later explored in \cite{Helbing1, Helbing2}. The space-time dynamics of the crowd is governed by a flow velocity vector field $V$ according to the transport equation 
	\begin{equation}\label{transport1}
		\partial_t u +\nabla \cdot (u\: V) =f.
	\end{equation}
	Here, $ u=u(t,x) $ is the density of individuals at time $t\geq 0$ and position $x\in \R^2$, which needs to accurately describe an admissible global distribution of the population, and $f$ is a given source. 
	
	The vector field $V$ takes into account the overall behaviour of the crowd (for example, the goal of reaching an exit or the avoidance of some danger) but neglects the local behaviour of pedestrians (who may, for example, be in a hurry, adapt their speed or try to avoid the crowd). To deal with local effects and following the predicting-correcting algorithms introduced in \cite{MRS1}, we consider a new vector field $W$ that will, in particular, consider congestion effects, thus obtaining the master equation
	\begin{equation}\label{transport2}
		\partial_t u +\nabla \cdot \left( W + u\: V \right) =f.
	\end{equation}
	
	The vector field $W=W(\nabla v)$ will be driven by the gradient of a potential $v \geq 0$ such that
	\begin{equation} \label{CC}
		v(u-1) =0,
	\end{equation}
	which is then a kind of Lagrange multiplier associated with the two-sided constraint $0\leq u \leq 1$. We can express this by requiring
	$$u \in \signp (v)\hbox{ and }v\geq 0,
	$$
	where $\signp$ denotes the maximal monotone graph given by
	$$
	\signp(r) := \left\{ 
	\begin{array}{cll}
		1 & \mathrm{if} & r>0\\  
		\displaystyle [0,1] & \mathrm{if} & r=0\\
		0 & \mathrm{if} & r<0. 
	\end{array}
	\right.
	$$ 
	
The linear case, corresponding to the choice
	\begin{equation} \label{optionmaury}
		W(\nabla v) = -\nabla v,
	\end{equation} 
	leads to the equation 
	$$
	\partial_t u - \Delta v +\nabla \cdot (u\: V) =f,
	$$
	which is relatively well-understood (see \cite{Ig2023}). Here, we explore the nonlinear case
	\begin{equation}\label{plap} 
		W(\nabla v) = -\left| \nabla v \right|^{p-2}\nabla v, \qquad 2<p<\infty,
	\end{equation}
	leading to the degenerate PDE 
	\begin{equation} 
		\partial_t u - \Delta_p v +\nabla \cdot (u\: V) =f, 
	\end{equation}
	and study the asymptotic limit problem obtained by taking $p \to \infty$.
	
	In the linear case $p=2$, congestion is modelled through linear diffusion and Brownian motion, and the crowd behaves like a Newtonian fluid. For $p>2$, the crowd behaves like a non-Newtonian shear-thickening or dilatant fluid, with the viscosity depending on the shear stress. 
	
	As we let $p$ approach infinity, we aim to capture a granular type of behaviour exhibited by the crowd, mirroring the well-established behaviour of sandpiles. Formally, the limiting problem aims to patch the transport equation with  
	\begin{equation}\label{plapinfty}
		W=-m\: \nabla v,\quad  \left|\nabla v\right|\leq 1,\quad m(\left|\nabla v\right|-1)=0,
	\end{equation}
	where $v$ is an unknown potential connected to the distance to the exit, supported in the congested region $[u =1]$, \textit{i.e.}, satisfying \eqref{CC}. The parameter $m$ is a Lagrange multiplier associated with the constraint $\left|\nabla v\right|\leq 1,$ which could be connected to the random movements of the individuals in the congested region (see \cite{EvRez} and \cite{IgStoch}). For a geometrical interpretation of $m$ in terms of the boundary curvature and the normal distance to the cut locus of the domain $\Omega$, see also \cite{Card1, Card2, CFV}. Moreover, rather than addressing the intricacies of the tangential gradient required due to the Radon measure $m$, as seen in \cite{BBS} for certain closely related particular cases, we adopt an equivalent formulation based on the variational characterization of the solution. To keep the presentation focused, we do not pursue a weak formulation (in the distributional sense) based on \eqref{plapinfty} in this work; this will likely be the subject of future investigations.

    Connecting the dynamics of a pedestrian moving towards a fixed target to that of sandpile particles moving towards the exit of a table is a plausible scenario introduced and studied numerically in \cite{EIJ}. In this model, the pile's height is linked to a \textit{potential} value so that higher potential areas have more particles (think of crowded zones).
	The self-organization of particles in a sandpile is a captivating natural phenomenon that has directly or indirectly inspired numerous physical models (see, for instance, \cite{BaPri1, BaPri2, DIg1, DIg2, Pr2, Pr}). Unlike the growth of a sandpile, where a source and gravity govern the dynamics, the movement in crowd motion is determined by the instantaneous movements of particles driven by the spontaneous velocity field $V$. Additionally, the approach could be formally grounded at the microscopic level by employing the stochastic sandpile model introduced by Evans and Rezakhanlou (see \cite{EvRez} and \cite{IgStoch}).  
	
	Imagine a grid of cubes (see Figure \ref{Figcubes}) representing pedestrians trying to reach an exit. Like a person, each cube can only move downhill (to a lower cube) randomly until it gets stuck. This creates a flow of \textit{pedestrian-cubes} similar to sand in a sandpile. Using an appropriate scaling of time and space, one would guess the resulting continuous dynamics follows a sandpile macroscopic flow to remedy the congestion. People (cubes) move downhill (following the gradient) but only when it is \textit{favourable} (think of a passage leading to the exit around a congested zone with a staircase offering sequential available positions). This movement is described by a \textit{flow} equation, where the flow is controlled by the potential's gradient. Indeed, in their pioneering work \cite{EvRez} (see also \cite{IgStoch}), Evans and Rezakhanlou study the case of a sandpile when the congestion constraint and the transport term are absent, \textit{i.e.}, for $v=u$ and $V=0$. They prove that the rescaled pile's height converges to the solution of a nonlinear sandpile dynamics governed by a flux $\Phi$ derived from a potential $z$, as expressed by 
	$$\Phi=-m\: \nabla z,$$
	where $z$, linked to the sandpile's height, satisfies the gradient constraint 
	$$
	\left|\nabla z\right|\leq 1,$$
	closely mirroring the discrete constraint on the cubes at the microscopic level. Additionally, $m\geq 0$ is an unknown parameter subject to the condition $m(\left|\nabla z\right|-1)=0$, which reflects the fact that particle movement towards the exit occurs only under favourable circumstances delineated by the gradient of $z.$ 	Hence, one can formally map the random cube movement to the behaviour of pedestrians in the congested regime, as depicted in the formal illustration of Figure \ref{Figcubes}.

	\begin{figure}[!t]
		\centering
		\includegraphics[width=11cm,height=5cm]{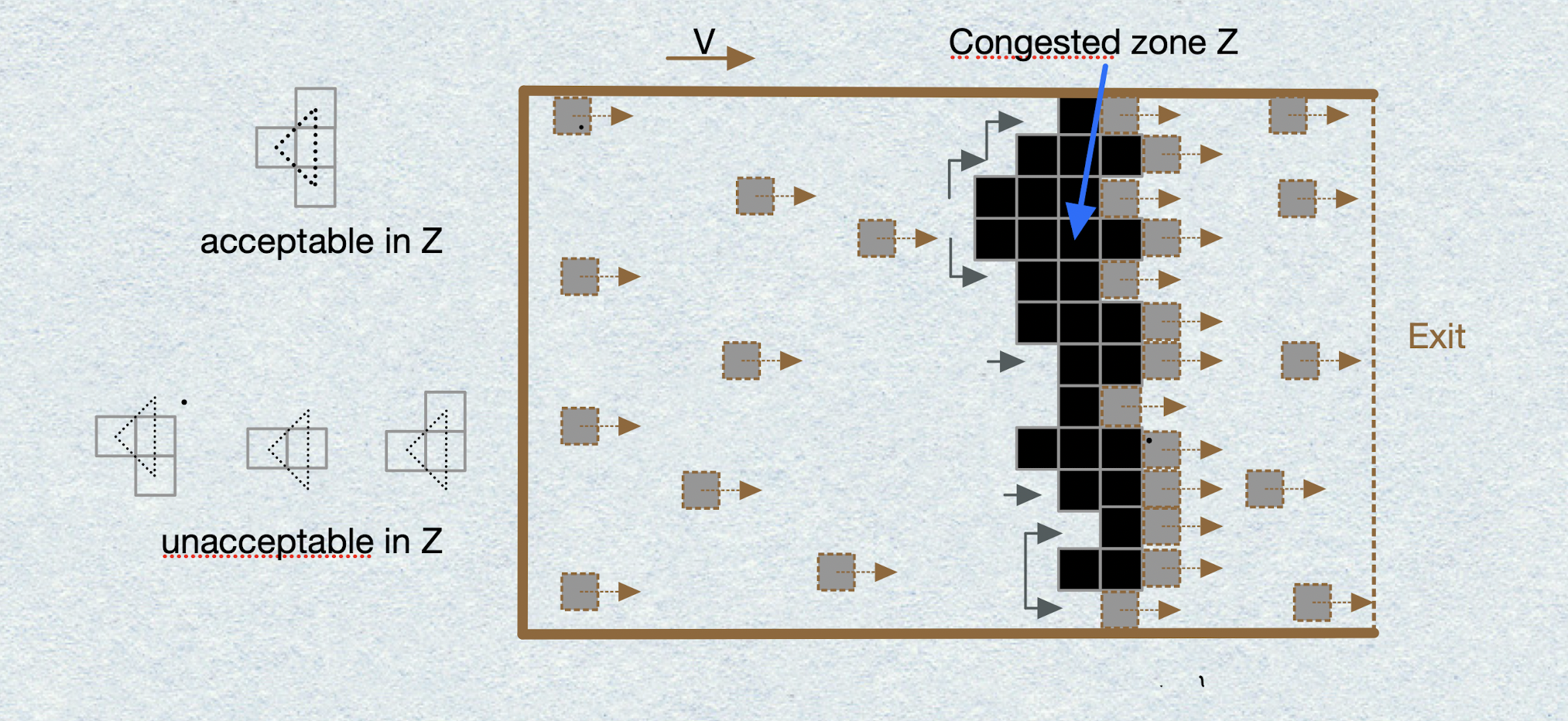}
		\caption{Toy pedestrian-cubes model}
		\label{Figcubes}
	\end{figure}

    The well-established micro-macro limit for sandpiles — rigorously derived by transitioning from a discrete spatio-temporal stochastic model at the cube scale to a continuous one — strongly motivates the use of the $p-$Laplacian to describe sand-grain-like granular dynamics. However, extending this approach to population dynamics remains challenging. The presence of a two-sided constraint on the density and the first-order transport term $V$ obstructs the adaptation of the techniques from \cite{EvRez} and \cite{IgStoch}, primarily due to difficulties with spatial and temporal scaling. The problem remains open and deserves further investigation, particularly within this context, which is highly relevant for applications in pedestrian dynamics.
    
	We conclude that \eqref{plap} and \eqref{plapinfty} provide two variants of macroscopic crowd motion models with hard congestion, aligning with the class of models introduced by Maury and collaborators (\textit{cf.} \cite{MRS1, MRS2, MRSV}). Unlike the linear scenario \eqref{optionmaury}, which represents the homogeneous random movement of pedestrians around the congested zone, these variants enable the natural handling (at the macroscopic level) of pedestrian movement, allowing them to occupy empty adjacent sites in the congested area towards the exit when possible, or to come to a halt if necessary. 
	
	The plan of the paper is the following: in Section \ref{aamr}, we gather some notation and the assumptions and state the main results; in Section \ref{Scontraction}, following \cite{Ig2023}, we establish the uniqueness of nonnegative weak solutions for the $p-$problem using the ideas of DiPerna-Lions on renormalization, and Kruzhkov's doubling and de-doubling techniques; in Section \ref{Sexistencep}, we prove the existence of a weak solution by employing nonlinear semigroup methods, building upon the $L^1-$contraction results from the previous section; finally, Section \ref{Slimitp} provides the proof of the convergence of solutions of the $p-$problem to the congested crowd motion problem, as the parameter $p$ approaches infinity; the appendix contains the proofs of some technical results used in Section \ref{Scontraction}. 
	
	\section{Assumptions and main results}\label{aamr}
	
	We assume that $\Omega \subset \R^N$ is a bounded open set, with a regular boundary, split into $\partial \Omega=\Gamma_D\cup \Gamma_N,$ such that $\Gamma_D\cap \Gamma_N=\emptyset$ and 
	$$
	\mathcal L^{N-1}(\Gamma_D)> 0.
	$$
	For $T>0$, we denote
	$$
	Q:=(0,T)\times \Omega; \qquad \Sigma_D:= (0,T)\times \Gamma_D; \qquad \Sigma_N:= (0,T)\times \Gamma_N.
	$$
	Given a source $f$, a velocity vector field $V$ and an initial datum $u_0$, we consider the problem of finding a pair of nonnegative functions $(u,v)$ such that
	\begin{equation}
		\label{cmef}
		\left\{  
		\begin{array}{lcl}
			\left.
			\begin{array}{l}
				\displaystyle \partial_t u - \nabla \cdot \left( \left|\nabla v\right|^{p-2}\nabla v - u  \: V \right) = f\\
				\\
				u\in \sign^+ (v)  
			\end{array}
			\right\}
			\   & \mathrm{in} & Q \\  
			\\
			\ \ v= 0  & \mathrm{on} & \Sigma_D\\ 
			\\
			\ \ \left( \left|\nabla v\right|^{p-2}\nabla v - u  \: V \right) \cdot \nu = 0  & \mathrm{on} & \Sigma_N\\  
			\\
			\ \ u (0)=u _0 & \mathrm{in} & \Omega,
		\end{array} 
		\right.
	\end{equation}
	
	\medskip
	
	\noindent where $p>2$ and $\nu$ is the outward unitary normal to $\partial \Omega$. 
	
\smallskip
	
Throughout the  paper, we assume
\begin{equation} \label{HypV}
		V \in  \left[ W^{1,p'}(\Omega) \right]^N, \qquad \nabla \cdot V \in L^\infty(\Omega),
\end{equation}
and  
\begin{equation} \label{HypV0}
V \cdot \nu \geq 0 \quad \mathrm{on}\ \Gamma_D  \qquad \mathrm{and} \qquad V \cdot \nu = 0 \quad \mathrm{on}\ \Gamma_N ,
\end{equation}
in the sense that 
\begin{equation}\label{HypVstg}
		\liminf_{h\to 0} \frac{1}{h} \int_{\{ x\in \Omega\: :\: d(x,\partial \Omega)< h\}} \xi \, V(x) \cdot \nu(\pi(x)) \, dx \geq   0,  
\end{equation}
for all  $0\leq \xi \in L^p(\Omega)$. Here, $d(.,\partial \Omega)$ is the Euclidean distance to the boundary of $\Omega$, and $\pi(x)$ denotes the projection of $x$ onto the boundary $\partial \Omega$. 

\medskip
Our first result concerns the existence and uniqueness of a weak solution to \eqref{cmef} in the sense we next precise. We denote 
$$
W^{1,p}_D(\Omega) : =\left\{ w\in W^{1,p}(\Omega) \, : \, w=0 \ \mathrm{on}\ \Gamma_D \right\}. 
$$

\begin{definition} \label{weaksol}
A couple $(u ,v)$ is a weak solution of \eqref{cmef} if 
$$(u ,v)\in  C \left( [0,T),L^1(\Omega) \right) \cap L^\infty(Q) \times  L^p \left(0,T;W^{1,p}_D(\Omega)\right),$$
$$ v\geq 0,\qquad u \in \signp(v),\mbox{ a.e. in } Q,$$ 
and
\begin{eqnarray*}
- \iint_Q u\: \xi \: \psi^\prime + \iint_Q \left( \left|\nabla v\right|^{p-2} \nabla v - u \:V \right) \cdot \nabla\xi \: \psi\\
=\iint_Q f \: \xi \: \psi + \int_\Omega u_0 \:\xi \: \psi(0), 
\end{eqnarray*}
for any $\xi \in W^{1,p}_D(\Omega)$ and $\psi\in \D \left( [0,T) \right)$. 	  
\end{definition}

\begin{theorem}\label{Theorem1}
For any $0\leq f\in L^{p'}(Q)$ and $u_0\in L^\infty(\Omega)$ such that 
$$0\leq u_0 \leq 1, \quad \mbox{a.e. in } \Omega,$$
the problem \eqref{cmef} has a unique weak solution in the sense of Definition \ref{weaksol}.
\end{theorem}
	
	\medskip 
	
	We next analyze the behaviour of the evolution problem \eqref{cmef} as the exponent $p$ approaches infinity. Since the pioneering work \cite{AEW} (see also \cite{Pr}), this limit represents a fundamental shift in the dynamics, revealing a critical connection between the long-term behaviour of the original $p-$Laplacian equation and the dynamics of grains in sandpile models (see also \cite{DIg1, DIg2, EFG, Pr2}). Indeed, letting $p\to\infty$ in the original equation 
	$$
	\frac{\partial z }{\partial t} - \nabla \cdot \left( \left|\nabla z\right|^{p-2}\nabla z  \right) = f \quad \text{in } Q,
	$$
	we obtain the limiting problem given by 
	\begin{equation} \label{limitp}
		\frac{\partial z }{\partial t} - \nabla \cdot \left( m\: \nabla z  \right) = f, \qquad \vert \nabla z\vert \leq 1\quad \text{in } Q, 
	\end{equation}
	where $m\geq 0$ is an unknown parameter that depends on the solution itself through the condition  
	$$
	m(\vert \nabla z\vert -1)=0 \quad \text{in } Q.
	$$
	So, formally, the limiting problem of \eqref{cmef}, as $p\to\infty$, may be given by 
	
	\begin{equation}
		\label{cmefinfty}
		\left\{  
		\begin{array}{lcl}
			\left.
			\begin{array}{l}
				\displaystyle \partial_t u  - \nabla \cdot \left( m\: \nabla v  - u  \: V \right) = f\\
				\\
				u\in \sign^+ (v), \quad	\vert \nabla v\vert \leq 1\\
				\\
				m\geq 0, \quad  m(\vert \nabla v\vert-1)=0 	  
			\end{array}
			\right\}
			\   & \mathrm{in} & Q \\  
			\\
			\ \ v= 0  & \mathrm{on} & \Sigma_D\\ 
			\\
			\ \ \left( m\: \nabla v - u  \: V \right) \cdot \nu = 0  & \mathrm{on} & \Sigma_N\\  
			\\
			\ \ u (0)=u _0 & \mathrm{in} & \Omega.
		\end{array} 
		\right.
	\end{equation} 
	
	\medskip
	
\noindent However, while $m$ is typically a Radon measure in similar settings (see, for instance, \cite{BBS} for the case of \eqref{limitp}), the gradient of $v$ requires a specialized approach called the \textit{tangential gradient} (see \cite{BBS} for details). To avoid this complexity, which we plan to explore further in future works, we will instead leverage an equivalent formulation based on the variational description of the solution (see also \cite{IgEquiv, IgEvol, DJ} for details regarding this equivalence in the context of \eqref{limitp}). Our focus will be on the characterization of the limit of the solutions to \eqref{cmef} using this variational formulation. We formally justify our choice as follows: considering the subgradient constraint in the diffusion operator governing the PDE in \eqref{cmefinfty}, we propose an integral formulation incorporating $v-\xi$ as a test function, for $0\leq \xi \in W_D^{1,\infty} (\Omega)$ and $| \nabla \xi | \leq 1$. Following this setup, it is straightforward to observe that  
$$
    \int_{\Omega}  m\: \nabla v    \cdot \nabla (v-\xi) = \int_{\Omega}  m\: (1- \nabla v    \cdot  \nabla \xi)  \geq 0.
$$
Furthermore, since $u\in \signp(v)$, we formally derive that     
$$ \int_\Omega \partial_t u\: (v-\xi) = \int_{[u=1]} \partial_t u\: v - \int_\Omega \partial_t u\: \xi  =  - \frac{d}{dt}\int_\Omega u\: \xi.$$ 
This suggests the viability of reformulating the problem into the variational form
$$- \frac{d}{dt}\int_\Omega u\: \xi \leq \int_\Omega   u  \: V \cdot \nabla \left( v - \xi \right) + \int_\Omega  f  \left( v - \xi \right), \quad \hbox{ in }\D'([0,T)),$$
according to the following definition.

\begin{definition} \label{varsol}
A couple $(u ,v)$ is a variational solution of \eqref{cmefinfty} if 
$$(u,v) \in  L^\infty(Q) \times L^q \left(0,T;W^{1,q}_{D}(\Omega) \right), \mbox{ for any } 1\leq q<\infty,$$
$$ v\geq 0,\qquad u \in \signp(v),\quad \left| \nabla v \right| \leq 1, \quad \mbox{a.e. in } Q,$$ 
and
\begin{eqnarray*}
\iint_{Q} u \: \xi \: \psi'(t) + \int_{\Omega} u_0 \: \xi \: \psi(0) \leq  \iint_{Q} u  \: V \cdot \nabla \left( v - \xi \right) \psi + \iint_{Q} f  \left( v - \xi \right) \psi,  
\end{eqnarray*}
for any $0\leq \xi \in W_D^{1,\infty} (\Omega)$ such that $|\nabla \xi | \leq 1$, a.e. in $\Omega$, and $0\leq \psi\in \D \left( [0,T) \right)$.  
\end{definition}
 
\begin{theorem}\label{Theorem2}
For any $0\leq f\in L^{p'}(Q)$ and $u_0\in L^\infty(\Omega)$ such that 
$$0\leq u_0 \leq 1, \quad \mbox{a.e. in } \Omega,$$ 
let $(u_p,v_p)$ be the weak solution of \eqref{cmef}. For subsequences that we relabel for convenience, we have  
$$u_p \rightharpoonup u \quad \mathrm{in} \quad L^\infty (Q)-\mbox{weak-}\ast;$$
$$v_p \rightharpoonup v  \quad \mathrm{in} \quad L^q \left(0,T;W^{1,q}(\Omega)\right)-\mathrm{weak},$$
as $p\to\infty$, where $(u,v)$ is a variational solution of the problem \eqref{cmefinfty} in the sense of Definition \ref{varsol}.
\end{theorem}

\begin{remark}
As previously noted, establishing the equivalence between the notion of solution involving the tangential gradient and the variational formulation we consider remains an open problem, one that we plan to explore in future work. Here, we remind the reader that this equivalence hinges on the subgradient operator governing the evolution being defined by $$v\to - \nabla \cdot  ( m\: \nabla v ),$$
where $m\geq 0$ and $\vert \nabla v\vert \leq 1$. It coincides with the subdifferential operator $v\to \partial I\!\!I_{K}$, where $K$ is given by 
$$K:=\Big\{z\in W^{1,p}_{D}(\Omega)\: :\:   \vert \nabla z\vert \leq 1\hbox{ a.e. in }\Omega \Big\}.$$ 
A function g belongs to $ \partial I\!\!I_{K} (v)$ if and only if 
$$\int_\Omega g\: (v-z)\geq 0,\quad \hbox{ for any }z\in K. $$  
However, working with the subgradient operator within \eqref{cmefinfty} requires additional tools and developments, which, while important, extend beyond the primary focus of this paper.
\end{remark}   
	 
\medskip	
	To close this section, we introduce some further notations to be used in the paper. Define, for each $h>0$, 
	\begin{equation}\label{xih}
		\xi_h(x) : =\frac{1}{h} \min \left\{ h, d(x,\partial \Omega) \right\} \qquad \mathrm{and} \qquad \nu_h(x)=-\nabla \xi_h (x), 
	\end{equation} 
	for $x\in \Omega$. The function $\xi_h \in H^1_0(\Omega)$ is regular (as smooth as the boundary) and concave, $0\leq \xi_h\leq 1$ and, for any $x\in \Omega$ such that $d(x,\partial \Omega)<h$,
	$$ 
	\nu_h(x) =  - \frac{1}{h} \nabla d(x,\partial \Omega).  
	$$ 
	In particular, for such $x$, we have  $h\nu_h(x) = \nu(\pi(x))$.  
	We denote 
	$$
	\nabla^{p-1} w := \left|\nabla  w \right|^{p-2} \nabla  w
	$$
	and let $\sign_0$ be the real discontinuous function defined in $\R$ by
	$$
	\sign_0(r) = \left\{ 
	\begin{array}{ccl}
		1 & \mathrm{if} & r>0 \\  
		0 & \mathrm{if} & r=0\\
		-1 & \mathrm{if} & r<0.
	\end{array}
	\right.
	$$   
	
	\section{$L^1-$contraction}\label{Scontraction}
	
	In this section, we focus first on the uniqueness and $L^1-$comparison principle for weak solutions. Following the approach developed in \cite{Ig2023}, we need a Kato's inequality, whose proof uses, in an essential way, the fact that weak solutions are also renormalized solutions \textit{à la} DiPerna-Lions. This is the object of the following result. 
	
	\begin{proposition}[Renormalised formulation] \label{prenormal}
		If $(u,v)$ is a weak solution of \eqref{cmef}, then
		$$	
		\partial_t \beta (u) - \Delta_p v  + V \cdot \nabla \beta(u) + u \: \nabla \cdot V \beta'(u) \leq f \: \beta'(u) \quad \textit{in}\ \D'(Q),
		$$
		for any $\beta \in \mathcal C^1(\RR)$ such that $\beta'\leq 1$ and $\beta'(1)=1$. 
	\end{proposition} 
	
	We postpone the proof of this proposition to the appendix.
	
	\begin{proposition}[Kato's inequality]\label{PKato}
		If $(u_1,v_1)$ and $(u_2,v_2)$ are two weak solutions of \eqref{cmef} associated with $f_1,f_2 \in L^1(Q)$, then there exists $\kappa\in L^\infty(Q)$ such that $\kappa\in \signp(u_1-u_2)$, a.e. in $Q$, and 
		$$
		\partial_ t	\left|u_1-u_2 \right|- \Delta_p  (v_1  +  v_2) + \nabla \cdot \left( \left| u_1-u_2 \right| \: V \right) 
		$$
		\begin{equation}\label{ineqkato}
			\hspace*{6cm} \leq  \kappa  (f_1-f_2)  \quad \mathit{in} \ \D^\prime (Q).
		\end{equation}
	\end{proposition}
	
	\begin{proof} 
		First, we see that if $(u,v)$ is a weak solution of \eqref{cmef}, then  
		$$
		\partial_t \left|u-k\right|- \Delta_p v + \nabla \cdot (\left|u-k\right|\:  V) + k \: \nabla \cdot V \: \hbox{Sign}_1 (u-k)   
		$$
		\begin{equation} \label{evolentropic+}
			\hspace*{6cm} \leq f \: \hbox{Sign}_1 (u-k) \quad \mathrm{in} \ \D^\prime (Q),		 			
		\end{equation}
		for any $k\leq 1$, where 
		$$
		\sign_1(r) = \left\{ 
		\begin{array}{ccl}
			1 & \mathrm{if} & r \geq 0 \\  
			-1 & \mathrm{if} & r<0.
		\end{array}
		\right.
		$$   	
		Indeed, it is enough to take in Proposition \ref{prenormal} 
		$$
		\beta_\epsilon(r)=\widetilde \heps (r+\eps-k), \quad r\in \R, 
		$$
		where
		$$
		\widetilde{\heps}(r) = \left\{ 
		\begin{array}{ccl}   
			r-{\eps}/{2}   & \mathrm{if} & r>\eps\\
			r^2/2\eps  & \mathrm{if} & \left|r\right|\leq \eps \\
			-r-{\eps}/{2}   & \mathrm{if} & r<-\eps ,
		\end{array}   
		\right.  
		$$
		and let $\eps\to 0$. Notice that, since $k\leq 1$, we have $\beta_\epsilon'(1)= \widetilde \heps' (1+\eps-k)= 1$ 
		and 
		$$
		\beta_\eps'(u) = \widetilde \heps'(u+\eps-k)\to \hbox{Sign}_1 (u-k),\quad \hbox{ as }\eps\to 0.
		$$ 
		
		The proof is now based on the doubling and de-doubling variables technique. Let us briefly revisit the arguments for the reader's convenience. Since $u_2(s,y)\leq 1$, we use the fact that $(u_1,v_1)$ satisfies  \eqref{evolentropic+}  with $k=u_2(s,y), $ to get  \begin{eqnarray*}
			\frac{d}{dt} \int_\Omega \left|u_1(t,x)-u_2(s,y) \right| \zeta(x,y)\, dx \hspace*{6cm}  \\
			+ \int_\Omega \left( \nabla_x^{p-1} v_1 (t,x) - \left| u_1(t,x)-u_2(s,y) \right| V(x) \right) \cdot \nabla_x \zeta (x,y)\, dx\\      
			+ \int_\Omega u_2(s,y) \left( \nabla_x \cdot V \right) \hbox{Sign}_1 \left( u_1(t,x)-u_2(s,y) \right) \zeta(x,y)\, dx\\
			\leq \int_\Omega f_1(t,x)  \, \hbox{Sign}_1 \left( u_1(t,x)-u_2(s,y) \right) \zeta(x,y)\, dx, \hspace*{3cm} 
		\end{eqnarray*} 	
	for any  $0\leq \zeta \in \D (\Omega\times \Omega)$, where    $\frac{d}{dt}$ is taken in $\D'(0,T).$ 
		Note that 
		$$ \int_\Omega \nabla_y^{p-1} v_2 (s,y) \cdot \nabla_x\zeta  \, dx =0,$$ 
		so that
		\begin{eqnarray*}
			\frac{d}{dt} \int_\Omega \left|u_1(t,x)-u_2(s,y) \right| \zeta \, dx \hspace*{6cm} \\
			+  \int_\Omega \left( \nabla_x^{p-1} v_1(t,x) + \nabla_y^{p-1} v_2(s,y) \right) \cdot \nabla_x\zeta \, dx\\
			-  \int_\Omega \left| u_1(t,x)-u_2(s,y) \right| V(x) \cdot \nabla_x\zeta \, dx \\   
			+ \int_\Omega u_2(s,y) \left( \nabla_x \cdot V \right) \hbox{Sign}_1 \left( u_1(t,x)-u_2(s,y) \right) \zeta \, dx \\
			\leq \int_\Omega f_1(t,x) \, \hbox{Sign}_1 \left( u_1(t,x)-u_2(s,y) \right) \zeta\, dx.  \hspace*{3cm} 
		\end{eqnarray*} 
		Denoting
		$$
		u(t,s,x,y):=u_1(t,x)-u_2(s,y),
		$$
		and integrating with respect to $y$, we obtain
		\begin{eqnarray*}
			\frac{d}{dt} \int_\Omega\!\int_\Omega \left|u(t,s,x,y) \right| \zeta \, dx dy  \hspace*{6cm} \\
			+ \int_\Omega\!\int_\Omega \left( \nabla_x^{p-1} v_1 (t,x) +\nabla_y^{p-1} v_2 (s,y) \right) \cdot \nabla_x\zeta\, dxdy \\
			-  \int_\Omega\!\int_\Omega \left| u(t,s,x,y) \right| V(x) \cdot \nabla_x\zeta \, dxdy \\   
			+ \int_\Omega\!\int_\Omega u_2(s,y) \left( \nabla_x \cdot V \right) \hbox{Sign}_1 \left( u(t,s,x,y) \right) \zeta\, dxdy\\
			\leq \int_\Omega\!\int_\Omega f_1(t,x) \, \hbox{Sign}_1 \left( u(t,s,x,y) \right) \zeta\, dxdy. \hspace*{3cm}   
		\end{eqnarray*} 
		
		On the other hand, using the fact that $(u_2,v_2)$ satisfies \eqref{evolentropic+} with $k=u_1(t,x)$, we have
		\begin{eqnarray*}
			\frac{d}{ds} \int_\Omega \left|u(t,s,x,y) \right| \zeta (x,y) \, dy \hspace*{6cm} \\
			+ \int_\Omega \left( \nabla_y^{p-1} v_2 (s,y) - \left| u(t,s,x,y) \right| V(y) \right) \cdot \nabla_y \zeta (x,y) \, dy \\
			- \int_\Omega u_1(t,x) \left( \nabla_y \cdot V \right) \hbox{Sign}_1 (u(t,s,x,y)) \: \zeta (x,y) \, dy \\   
			\leq - \int_\Omega f_2(s,y) \, \hbox{Sign}_1 ( u(t,s,x,y) ) \: \zeta (x,y) \, dy,  \hspace*{3cm} 
		\end{eqnarray*}
		where, again, $\frac{d}{ds}$ is taken in $\D'(0,T)$. Working in the same way, we get
		\begin{eqnarray*}
			\frac{d}{ds} \int_\Omega\!\int_\Omega \left|u(t,s,x,y) \right| \zeta \, dx dy  \hspace*{6cm} \\
			+  \int_\Omega\!\int_\Omega \left( \nabla_x^{p-1} v_1 (t,x) + \nabla_y^{p-1} v_2 (s,y) \right) \cdot \nabla_y\zeta\, dx dy \\
			-   \int_\Omega\!\int_\Omega \left|u(t,s,x,y) \right| V(y) \cdot \nabla_y\zeta \, dxdy \\   
			- \int_\Omega\!\int_\Omega  u_1(t,x) \left( \nabla_y \cdot V \right)  \hbox{Sign}_1 \left( u(t,s,x,y) \right) \zeta\, dxdy\\
			\leq  - \int_\Omega\!\int_\Omega f_2(s,y) \, \hbox{Sign}_1 \left( u(t,s,x,y) \right) \zeta\, dxdy. \hspace*{3cm}   
		\end{eqnarray*}
		Adding both inequalities, we obtain
		\begin{eqnarray}
			\left( \frac{d}{dt} + \frac{d}{ds} \right) \int_\Omega\!\int_\Omega \left|u(t,s,x,y) \right| \zeta \, dx dy  \hspace*{5cm} \label{formdoubling1} \\
			+  \int_\Omega\!\int_\Omega  \left( \nabla_x^{p-1} v_1 (t,x) +\nabla_y^{p-1} v_2 (s,y) \right) \cdot \left( \nabla_x + \nabla_y \right) \zeta\, dx dy \nonumber\\
			-   \int_\Omega\!\int_\Omega \left|u(t,s,x,y) \right| \left( V(x) \cdot \nabla_x\zeta + V(y) \cdot \nabla_y\zeta \right) dxdy \nonumber\\   
			+ \int_\Omega\!\int_\Omega \left( u_2(s,y) \left( \nabla_x \cdot V \right) - u_1(t,x) \left( \nabla_y \cdot V \right)   \right) \hbox{Sign}_1 \left( u(t,s,x,y) \right) \zeta\, dxdy\nonumber\\
			\leq  \int_\Omega\!\int_\Omega \left( f_1(t,x) - f_2(s,y) \right) \hbox{Sign}_1 \left( u(t,s,x,y) \right) \zeta\, dxdy, \hspace*{2cm}   \nonumber 
		\end{eqnarray} 
		where $\frac{d}{dt}+\frac{d}{ds}$ is taken in $\D'\left( (0,T)\times (0,T) \right)$. 
		
		We can now de-double the variables $t$ and $s$, as well as $x$ and $y$, by taking as usual the sequences of test functions  
		$$  
		\psi_\eps(t,s) = \psi\left( \frac{t+s}{2}\right)   \rho_\eps \left( \frac{t-s}{2}\right)
		$$
		and
		$$
		\zeta_\lambda  (x,y) = \xi \left( \frac{x+y}{2}\right) \delta_\lambda \left( \frac{x-y}{2}\right),
		$$
		for any $t,s \in (0,T)$ and $x,y \in \Omega$. Here, $\psi, \xi \in \D(\Omega)$, and $\rho_\eps, \delta_\lambda$ are sequences of standard mollifiers in $\R$ and $\R^N$, respectively. 
		Observe that 
		$$  
		\left( \frac{d}{dt} + \frac{d}{ds} \right) \psi_\eps(t,s) = \rho_\eps \left( \frac{t-s}{2}\right) \psi^\prime \left( \frac{t+s}{2} \right)   
		$$ 
		and 
		$$
		\left( \nabla_x+ \nabla_y \right) \zeta_\lambda (x,y) = \delta_\lambda \left( \frac{x-y}{2}\right)\nabla \xi \left( \frac{x+y}{2}\right).  
		$$
		Moreover, for any $h\in L^1 \left( (0,T)^2 \times \Omega^2 \right)$ and $\Phi \in L^1 \left( (0,T)^2 \times \Omega^2 \right)^N$,   we have 
		$$
		\lim_{\lambda \to 0} \lim_{\eps \to 0} \int_0^T\!\!\!\int_0^T\!\!\!\int_\Omega\!\int_\Omega h(t,s,x,y)\, \zeta_\lambda(x,y)\, \psi_\eps(t,s)  \, ds dt dx dy
		$$
		$$
		= \int_0^T\!\!\!\int_\Omega  h(t,t,x,x)\, \xi(x)\,  \psi(t) \, dtdx,
		$$
		
		$$
		\lim_{\lambda \to 0} \lim_{\eps \to 0} \int_0^T\!\!\!\int_0^T\!\!\!\int_\Omega\!\int_\Omega h(t,s,x,y)\, \zeta_\lambda(x,y)\, \left( \frac{d}{dt} + \frac{d}{ds} \right) \psi_\eps(t,s) \, dsdtdxdy  
		$$
		$$
		= \int_0^T\!\!\!\int_\Omega h(t,t,x,x)\, \xi(x)\, \psi^\prime (t)\,  dtdx
		$$ 
		and		 		
		$$
		\lim_{\lambda \to 0} \lim_{\eps \to 0} \int_0^T\!\!\!\int_0^T\!\!\!\int_\Omega\!\int_\Omega \Phi(t,s,x,y) \cdot  (\nabla_x + \nabla_y) \, \zeta_\lambda (x,y) \, \psi_\eps (t,s) \, dsdtdxdy  
		$$
		$$
		= \int_0^T\!\!\!\int_\Omega \Phi(t,t,x,x) \cdot \nabla \xi (x)\, \psi(t)\, dtdx.
		$$
		Thus, replacing $\zeta$ in \eqref{formdoubling1} by $\zeta_\lambda$, testing with $\psi_\eps$ and letting $\eps\to 0$ and $\lambda \to 0$, we obtain (see, for instance, \cite{Ig2023})
		\begin{eqnarray*}
			\frac{d}{dt} \int_\Omega \left| u_1-u_2 \right| \, \xi \, dx + \int_\Omega \left( \nabla^{p-1} v_1 + \nabla^{p-1} v_2 \right) \cdot \nabla \xi \, dx \\
			- \int_\Omega \left| u_1-u_2 \right| \left( V \cdot \nabla \xi - \left( \nabla \cdot V \right) \xi \right) dx  \\
			\leq  \int_\Omega \kappa \, (f_1-f_2) \, \xi \, dx + \int_\Omega \left| u_1-u_2 \right| \left( \nabla \cdot V \right) \xi  \, dx,   
		\end{eqnarray*}
		where $\frac{d}{dt}$ is taken in $\D'(0,T)$. We conclude, as desired, that
		\begin{eqnarray*}
			\frac{d}{dt} \int_\Omega \left| u_1-u_2 \right| \, \xi \, dx + \int_\Omega \nabla^{p-1} \left( v_1 + v_2 \right) \cdot \nabla \xi \, dx \\
			- \int_\Omega \left| u_1-u_2 \right| \, V \cdot \nabla \xi \, dx  \\
			\leq  \int_\Omega \kappa \, (f_1-f_2) \, \xi \, dx,   \quad \hbox{ in }\D'(0,T).
		\end{eqnarray*}
	\end{proof}
	
	The idea behind the proof of the next theorem is to consider the sequence of test functions $\xi_h$ given by \eqref{xih} in Kato's inequality and let $h\to 0,$ to get the contraction inequality \eqref{contraction}. 
	
	\begin{theorem} \label{compcmef}
		If $(u_1,v_1)$ and $(u_2,v_2)$ are two weak solutions of \eqref{cmef} associated with $f_1,f_2 \in L^1(Q)$, respectively, then there exists $\kappa \in L^\infty(Q)$ such that $\kappa\in \signp(u_1-u_2)$, a.e. in $Q$, and 
		\begin{equation}\label{contraction}
			\frac{d}{dt} \int_\Omega \left|  u_1-u_2 \right| dx  \leq \int_\Omega \kappa \left( f_1-f_2 \right) dx, \quad \mathit{in} \ \D'(0,T). 
		\end{equation}
	\end{theorem}
	
	\begin{proof}
		Observe that, for $\xi_h$ given by \eqref{xih}, we have
		$$
		\frac{d}{dt}	\int_\Omega \left| u_1-u_2\right| dx - \int_\Omega \kappa \left( f_1-f_2 \right) dx 
		$$
		$$
		=\lim_{h\to 0} \ \underbrace{\frac{d}{dt} \int_\Omega \left| u_1-u_2\right|  \, \xi_h \: dx  -\int_\Omega  \kappa \left( f_1-f_2 \right) \, \xi_h \, dx}_{I(h)}.
		$$
		Taking $\xi_h$ as a test function in \eqref{ineqkato}, we obtain
		\begin{eqnarray*}
			I(h) & \leq & - \int_\Omega \left( \nabla^{p-1} v_1 + \nabla^{p-1} v_2 - \left| u_1-u_2\right| V \right) \cdot \nabla \xi_h \, dx \\ 
			& \leq & - \int_\Omega \left( \nabla^{p-1} v_1 + \nabla^{p-1} v_2\right) \cdot \nabla \xi_h \, dx - \int_\Omega \left| u_1-u_2\right| V \cdot \nu_h(x) \, dx. 
		\end{eqnarray*}	
		On the other hand, thanks to \eqref{eqp+1}, we see that, for each $i=1,2$, for any $0\leq \psi\in \D(0,T),$ we have
		\begin{eqnarray*}
			\int_0^T\!\!\!\int_\Omega  \nabla^{p-1} v_i \cdot \nabla \xi_h \, \psi\, dtdx & = & - \int_0^T\!\!\!\int_\Omega \nabla^{p-1} v_i  \cdot \nabla \left( 1- \xi_h \right) \psi\, dtdx \\ 
			& \geq & \int_0^T\!\!\!\int_\Omega \left( \nabla \cdot V - f_i \right) \left( 1- \xi_h \right) \sign_0(v_i)\, \psi\, dtdx.   
		\end{eqnarray*} 
		Letting $h\to 0$ and using the fact that $\xi_h \to 1$ in $L^\infty(\Omega)-\hbox{weak}^*$,  we deduce that    
		$$ 	
		\liminf_{h\to 0} \int_0^T\!\!\!\int_\Omega \nabla^{p-1} v_i \cdot \nabla \xi_h \, \psi\, dtdx \geq 0.
		$$ 
		Coming back to $I(h),$ we get   
		$$
		\lim_{h\to 0}I(h) \leq - \lim_{h\to 0} \int_\Omega \left|u_1-u_2 \right| V \cdot \nu_h(x) \, dx \leq  0,		
		$$
		using assumption \eqref{HypVstg}. Thus, we obtain \eqref{contraction}.
		
	\end{proof}

	An immediate consequence of Theorem \ref{compcmef} is the uniqueness of a solution for \eqref{cmef}. 
	\begin{corollary}\label{Cuniq}
		Under the assumptions of Theorem \ref{Theorem1}, the problem \eqref{cmef} has at most one solution. 
	\end{corollary}
	
	\section{Existence for the evolution problem}\label{Sexistencep}
	
	The proof of the existence of a solution to \eqref{cmef} will be carried out in the framework of nonlinear semigroup theory in $L^1(\Omega)$. We consider the stationary problem, related to the Euler implicit discretization scheme of the evolution problem \eqref{cmef} 
	
	\smallskip
	
	\begin{equation} 
		\label{st}
		\left\{  
		\begin{array}{lcl}
			\left.
			\begin{array}{l}
				u - \lambda\: \Delta_p v + \lambda\: \nabla \cdot (u  \: V)=f\\
				\\
				u\in \sign^+ (v) 
			\end{array}
			\right\}
			\   & \mathrm{in} & \Omega \\  
			\\
			\ \ v= 0  & \mathrm{on} & \Gamma_D\\ 
			\\
			\ \ \left( \nabla^{p-1} v - u  \: V \right) \cdot \nu = 0  & \mathrm{on} & \Gamma_N,
		\end{array} 
		\right.
	\end{equation}
	where $f \in L^2(\Omega)$ and $\lambda >0$ are given.  
	
	\begin{definition}
		A couple $(u,v) \in  L^\infty(\Omega) \times W^{1,p}_D(\Omega)$ is a weak solution of \eqref{st} if $u \in \signp(v)$, a.e. in $\Omega$, and
		$$
		\int_\Omega u \:\xi \, dx + \lambda\:\int_\Omega \nabla^{p-1} v \cdot \nabla\xi \, dx - \lambda\:\int_\Omega  u \: V\cdot \nabla \xi \, dx= \int_\Omega f\: \xi \, dx ,\quad \forall \, \xi\in W^{1,p}_D(\Omega).
		$$
	\end{definition}		 
	
	\bigskip
	
	As a consequence of  Theorem \ref{compcmef}, we can deduce the following result. 
	
	\begin{corollary}\label{ccontractionst}
		If $(u _1,v_1)$ and $(u _2,v_2)$ are two solutions of \eqref{st} associated with $f_1, f_2 \in L^1(\Omega)$, respectively, then
		$$ 
		\left|u_1-u_2\right|_{1} \leq \left|f_1-f_2\right|_{1}.
		$$
	\end{corollary}
	
	\begin{proof}
		This is a simple consequence of the fact that if the (independent of $t$) couple $(u,v)$ is a weak solution of \eqref{st}, then it can be thought out as a time-independent solution of the evolution problem \eqref{cmef} with $f$ replaced by $f-u$ (which is also independent of $t$).  
		
	\end{proof}
	
	\bigskip We will consider in the sequel $\lambda =1$, the changes being obvious in the general case $\lambda >0$.   For $\eps >0$, let
	$$
	\heps(r) = \left\{ 
	\begin{array}{ccl}   
		1   & \mathrm{if} & r>\eps\\
		r/\eps  & \mathrm{if} & |r| \leq \eps \\
		-1   & \mathrm{if} & r<-\eps 
	\end{array}   
	\right.  
	$$
	
	\smallskip
	
	\noindent and consider the regularized problem
	
	\smallskip
	
	\begin{equation} 
		\label{psteps}
		\left\{  
		\begin{array}{lcl}
			\left.
			\begin{array}{l}
				u_\eps -  \Delta_p v_\eps +   \nabla \cdot (u_\eps  \: V)=f\\
				\\
				u_\eps =\heps  (v_\eps) 
			\end{array}
			\right\}
			\   & \mathrm{in} & \Omega \\  
			\\
			\ \ v_\eps= 0  & \mathrm{on} & \Gamma_D\\ 
			\\
			\ \ \left( \nabla^{p-1} v_\eps - u_\eps  \: V \right) \cdot \nu = 0  & \mathrm{on} & \Gamma_N.
		\end{array} 
		\right.
	\end{equation}
	
	\medskip
	
	\noindent Observe that, for any $\eps>0$, $ \left|\heps\right|\leq 1$, $\heps$ is Lipschitz continuous and it satisfies
	$$
	(I+\heps)^{-1} (r) \longrightarrow (I+\signp)^{-1} (r), \quad \mathrm{as} \ \eps\to 0,  \quad \hbox{for any} \ r\in  \R, 
	$$
	\textit{i.e.}, $\heps$ converges to $\signp$ in the sense of the resolvent, which is equivalent to the convergence in the sense of the graph (cf. \cite{Br}). 
	
	The following result establishes the existence for the regularized problem and, through a passage to the limit, the existence for \eqref{st}.
	
	\begin{proposition}\label{extsst} 
		For any $f\in L^{p'}(\Omega)$ and $\eps>0$, problem \eqref{psteps} has a weak solution $(u_\eps,v_\eps)$, in the sense that  $v_\eps\in W^{1,p}_D(\Omega)$, $ u_\eps=\heps(v_\eps)$, a.e. in $\Omega$, and
		\begin{equation}\label{weakeps}
        \begin{array}{l}
        \int_\Omega u_\eps \:\xi \, dx + \int_\Omega \nabla^{p-1} v_\eps \cdot \nabla\xi \, dx - \int_\Omega  u_\eps \: V\cdot \nabla \xi \, dx  \\  \\ 
			\quad \quad \quad  = \int_\Omega f\: \xi \, dx ,\quad \forall \, \xi\in W^{1,p}_D(\Omega).
		\end{array}\end{equation}
		Moreover, as $\eps\to 0,$ we have 
		\begin{eqnarray}
			\heps (v_\eps) & \longrightarrow  & u \quad \mathit{in} \ L^\infty(\Omega)-\mathit{weak}^\star, \label{convuepss}\\
			v_\eps & \longrightarrow  & v \quad \mathit{in} \  W^{1,p}_D(\Omega)-\mathit{weak} \label{convpepss} 
		\end{eqnarray}
		and  $(u,v)$ is the weak solution of \eqref{st}. Finally, if $f\geq 0 $ then $u,v \geq 0$, a.e. in $\Omega$.
	\end{proposition} 
	
	\begin{proof}   
		The existence of a solution for \eqref{psteps} is standard, but for completeness and the reader's convenience, we reproduce the main arguments. 
		
		Let us denote the topological dual space of $W^{1,p}_D(\Omega)$ by $\left[ W^{1,p}_D(\Omega) \right]^\star$ and the associated duality bracket by $ \langle \cdot , \cdot \rangle$. Observe that the operator 
		$$A_\eps \, : \, W^{1,p}_D(\Omega) \longrightarrow  \left[ W^{1,p}_D(\Omega) \right]^\star,$$ 
		defined, for $\xi\in W^{1,p}_D(\Omega)$, by
		$$ 
		\langle A_\eps v,\xi \rangle =  \int_{\Omega} \heps (v)\, \xi\, dx + \int_{\Omega} \nabla^{p-1} v \cdot \nabla \xi \, dx -\int_{\Omega} \heps (v) \, V\cdot \nabla \xi \, dx, 
		$$
		is bounded and weakly continuous. Moreover, $A_\eps$ is coercive since, for any $u\in W^{1,p}_D(\Omega),$ we have
		\begin{eqnarray*}
			\langle A_\eps v,v \rangle & = & \int_{\Omega} \heps (v)\, v\, dx + \int_{\Omega} \left| \nabla v \right|^p dx -\int_{\Omega} \heps (v) \, V\cdot \nabla v \, dx,\\ 
			& \geq & \int_{\Omega} \left| \nabla v \right|^p dx - \int_{\Omega} \left| V \right|  \left|  \nabla v \right| dx \\
			& \geq &  \frac{1}{p^\prime}  \int_{\Omega} \left| \nabla v \right|^p dx - \frac{1}{p^\prime} \int_{\Omega} \left|  V\right|^{p^\prime}   \: dx,
		\end{eqnarray*}
		using Young's inequality. Thus, for any $f\in \left[ W^{1,p}_D(\Omega) \right]^\star \supset L^{p^\prime} (\Omega)$, the problem $A_\eps v=f$ has a solution $v_\eps \in W^{1,p}_D(\Omega).$  
		
		To pass to the limit as $\eps\to 0$, we first note that 
		\begin{equation} \label{hestimate0}
			\int_{\Omega} \left| \nabla v_\eps \right|^p dx  \leq  C (N,p,\Omega) \left( \int_{\Omega} \left| V  \right|^{p^\prime} dx  + \int_{\Omega} \left| f  \right|^{p^\prime} dx \right).
		\end{equation}
		Indeed, taking $ v_\eps$ as a test function, we have
		$$\int_{\Omega} u_\eps  v_\eps \, dx + \int_{\Omega} \left| \nabla  v_\eps \right|^p dx =  \int_{\Omega} u_\eps\, V\cdot \nabla  v_\eps  \, dx + \int_{\Omega} f  v_\eps \, dx.$$
		Using Young's inequality and the fact that $\left| u_\eps \right| =\left| \heps(v_\eps) \right| \leq 1$, we obtain
		$$	\int_\Omega u_\eps\, V\cdot \nabla  v_\eps \, dx \leq \frac{1}{p^{\prime}}  \int_\Omega \left| V\right|^{p^{\prime}} + \frac{1}{p} \int_\Omega \left| \nabla v_\eps \right|^{p} dx  $$
		and, by combining Poincar\'e's with Young's inequalities, also
		$$\int_\Omega f v_\eps \, dx \leq  \frac{C}{p^{\prime}} \int_\Omega \left| f\right|^{p^{\prime}}  + \frac{1}{p} \int_\Omega \left| \nabla v_\eps \right|^{p} dx.$$
		Using the fact that $u_\eps v_\eps\geq 0$, we deduce \eqref{hestimate0}.  
		
		Now, it is clear that the  sequences $v_\eps$ and $u_\eps=\heps (v_\eps)$ are bounded, respectively, in $W^{1,p}_D(\Omega)$ and in $L^\infty(\Omega)$. Thus, there exists a subsequence (that we denote again by $v_\eps$) such that \eqref{convuepss} and \eqref{convpepss} are fulfilled. In particular, using a monotonicity argument (see, for instance, \cite{Br}), this implies that $u\in \signp(v)$, a.e. in $\Omega$, and, letting $\eps\to 0$ in \eqref{weakeps}, we obtain that $(u,v)$ is a weak solution of \eqref{st}. 
        
For the last part of the proposition, it is enough to prove that, for each $\eps >0,$  $u_\eps\geq 0$ and $v_\eps\geq 0$, a.e. in $\Omega$. With this in mind, we consider $\xi=\mathcal{H}_\sigma(v_\eps^-)$, for an arbitrary $\sigma>0$, as a test function in \eqref{weakeps}, obtaining
\begin{eqnarray*}
    \int_\Omega u_\eps \:\mathcal{H}_\sigma(v_\eps^-) + \int_\Omega \nabla^{p-1} v_\eps \cdot \nabla\mathcal{H}_\sigma(v_\eps^-) - \int_\Omega  u_\eps \: V\cdot \nabla \mathcal{H}_\sigma (v_\eps^-)   \\ 
	= \int_\Omega f\: \mathcal{H}_\sigma(v_\eps^-).
\end{eqnarray*}
This implies  
\begin{eqnarray*}
    -\int_\Omega u_\eps^- \:\mathcal{H}_\sigma(v_\eps^-)  - \int_\Omega \vert \nabla v_\eps^- \vert^p \:  \mathcal{H}_\sigma'(v_\eps^-) + \int_\Omega  u_\eps^-  \: V\cdot \nabla \mathcal{H}_\sigma (v_\eps^-)    \\ 
	= \int_\Omega f\: \mathcal{H}_\sigma(v_\eps^-),
\end{eqnarray*}
and so 
\begin{equation}  \label{interm1}
         \int_\Omega u_\eps^- \:\mathcal{H}_\sigma(v_\eps^-)  -  \int_\Omega  u_\eps^-  \: V\cdot \nabla \mathcal{H}_\sigma (v_\eps^-)    \leq - \int_\Omega f\: \mathcal{H}_\sigma(v_\eps^-)  \leq 0 ,
\end{equation}  
where we used the fact that $f\geq 0.$  Observe that 
\begin{eqnarray*}
            \int_\Omega  u_\eps^-  \: V\cdot \nabla \mathcal{H}_\sigma (v_\eps^-)  &=& \int_\Omega  \heps(v_\eps^-)  \: V\cdot \nabla \mathcal{H}_\sigma (v_\eps^-)  \\
            &=& \int_\Omega    \mathcal{H}_\sigma' (v_\eps^-) \:    \heps(v_\eps^-)  \: V\cdot \nabla   v_\eps^- \\
            &=& \frac{1}{\sigma}\int_{[v_\eps^- <\sigma]}        \heps(v_\eps^-)  \: V\cdot \nabla   v_\eps^-  \ \longrightarrow 0, \hbox{ as }\sigma\to 0, 
\end{eqnarray*}
where we used the fact that $r\to \heps(r)$ is  Lipchitz continuous.  So, letting $\sigma\to 0$ in \eqref{interm1}, we obtain 
$$\int_\Omega u_\eps^-   \leq 0,$$
which implies that $u_\eps^- =0$, a.e. in $\Omega$, and then also $v_\eps^-   =0$, a.e. in $\Omega$. Thus, $u_\eps$ and $v_\eps$ are nonnegative a.e. in $\Omega.$
     
\end{proof}
	
	To prove the existence of a weak solution to \eqref{cmef}, we fix $f\in L^{p'}(Q)$ and, for an arbitrary $0<\eps\leq \eps_0$ and $n\in \N$ such that $n\eps=T$, we consider the sequence $(u_i,v_i)$ given by the $\eps-$Euler implicit scheme associated with \eqref{cmef}, namely
	\begin{equation} 
		\label{sti}
		\left\{  
		\begin{array}{lcl}
			\left.
			\begin{array}{l}
				u_{i+1} - \eps \Delta_p v_{i+1} + \eps \, \nabla \cdot (u_{i+1} V) = u_{i} +\eps f_i\\
				\\
				u_{i+1} \in \sign^+(v_{i+1}) 
			\end{array}
			\right\}
			\   & \mathrm{in} & \Omega \\  
			\\
			\ \ v_{i+1} = 0  & \mathrm{on} & \Gamma_D\\ 
			\\
			\ \ \left( \nabla^{p-1} v_{i+1} - u_{i+1}  \: V \right) \cdot \nu = 0  & \mathrm{on} & \Gamma_N.
		\end{array} 
		\right.
	\end{equation}
	
	\smallskip
	
	\noindent where, for each $i=0,\ldots, n-1$, $f_i$ is given by
	$$f_i = \frac{1}{\eps} \int_{i\eps}^{(i+1)\eps} f(s)\, ds, \quad \hbox{a.e. in } \Omega. $$
	Now, for a given $\eps-$time discretization $0=t_0<t_1<\ldots<t_n=T$, satisfying $t_{i+1}-t_i = \eps$, we define the $\eps-$approximate solution by
	$$u_\eps:= \sum_{i=0}^{n-1 } u_i \, \chi_{[t_i,t_{i+1})} \qquad \hbox{and} \qquad v_\eps:= \sum_{i=1}^{n -1} v_i \, \chi_{[t_i,t_{i+1})}.$$
	
	Due to Proposition \ref{extsst} and the general theory of evolution problems governed by accretive operators (see, for instance, \cite{Benilan, Barbu}), we define the operator $\A$ in $L^1(\Omega)$ by $\mu\in \A(z)$ if, and only if, $\mu, z\in L^1(\Omega)$ and $z$ is a solution of the problem 
	$$
	\left\{  
	\begin{array}{lcl}
		\left.
		\begin{array}{l}
			- \Delta_p v +  \nabla \cdot (z  \: V)=\mu\\
			\\
			z \in \sign^+(v) 
		\end{array}
		\right\}
		\   & \mathrm{in} & \Omega \\  
		\\
		\ \ v = 0  & \mathrm{on} & \Gamma_D\\ 
		\\
		\ \ \left( \nabla^{p-1} v - z  \: V \right) \cdot \nu = 0  & \mathrm{on} & \Gamma_N.
	\end{array} 
	\right.
	$$
	
	\smallskip
	
	\noindent in the sense that $z\in L^\infty(\Omega)$ and there exists $v \in  W^{1,p}_D(\Omega)$ satisfying $z \in \signp(v)$, a.e. in $\Omega$, and 
	$$\int_\Omega \nabla^{p-1} v \cdot \nabla\xi \, dx - \int_\Omega  z \:  V\cdot \nabla \xi \, dx = \int_\Omega \mu \: \xi \, dx, \quad \forall \:  \xi \in W^{1,p}_D(\Omega).$$
	
	As a consequence of Corollary \ref{ccontractionst}, we know that the operator $\A$ is accretive in $L^1(\Omega)$. Moreover, we have 
	$$\overline{\D(A)}= \left\{ u\in L^\infty(\Omega)\: : \:  \left| u \right| \leq 1, \hbox{ a.e. in } \Omega \right\}.$$ 
	It then follows from the general theory of nonlinear semigroups governed by accretive operators (see, for instance, \cite{Barbu}) that, as $\eps\to 0$, 
	\begin{equation}\label{convueps}
		u_\eps \longrightarrow u, \quad \hbox{ in } \C \left( [0,T), L^1(\Omega) \right),
	\end{equation}
	and $u$ is the so-called \textit{mild solution} of the evolution problem
	\begin{equation}\label{cauchy}
		\left\{\begin{array}{ll}
			u_t + \A u \ni f & \hbox{in }(0,T)\\  
			\\
			u(0)=u_0.
		\end{array}  \right.
	\end{equation}
	
	To complete the proof of the existence for problem \eqref{cmef}, we show that the mild solution $u$ is, in fact, the solution of \eqref{cmef}. More precisely, we prove the following result. 
	
	\begin{theorem}\label{existcmef}
		For any non-negative $ f\in L^{p^\prime}(Q)$ and $u_0\in L^\infty(\Omega)$, the mild solution of \eqref{cauchy} is a  solution of \eqref{cmef}, in the sense that there exists $v\in L^p(0,T;W^{1,p}_D(\Omega)$ such that the couple $(u,v)$ solves the problem \eqref{cmef} in the sense of Theorem \ref{Theorem1}. Moreover, we have 
			\begin{equation}\label{newestimate}
				\iint_Q \left|\nabla v \right|^p   \leq   \iint_Q  f\: v + \iint_Q V\cdot \nabla v.
			\end{equation}
	\end{theorem}
	
	To this aim, we use the limit of the sequence $v_\eps$, given by the $\eps-$approxima\-te solution. 
	
	\begin{lemma}\label{lmildweak}
		We have, as $\eps\to 0$, 
		$$v_\eps \longrightarrow v,\quad \hbox{in } L^p \left( 0,T; W^{1,p}_D(\Omega) \right),$$
	$v$ satisfies \eqref{newestimate}  	and $(u,v)$ is a weak solution of \eqref{cmef}.   
	\end{lemma}
	
\begin{proof}
Due to Proposition \ref{extsst}, the  sequence $(u_i,v_i)$ given by \eqref{sti} is well defined in $L^\infty(\Omega)\times W^{1,p}_D(\Omega)$, $v_i \geq 0$, $u_i\in \signp(v_i)$, and 
\begin{equation}\label{stwsi}
\begin{array}{l} 		
\int_\Omega u_{i+1} \:\xi \, dx + \eps \int_\Omega \nabla^{p-1} v_{i+1} \cdot  \nabla\xi \, dx - \eps \int_\Omega u_{i+1} \:  V\cdot \nabla \xi \, dx   \\  \\ 
\quad \quad \quad  = \int_\Omega u_i \: \xi \, dx + \eps \int_\Omega f_i \: \xi \, dx ,\quad \forall \, \xi\in W^{1,p}_D(\Omega).
\end{array}
\end{equation}
Taking $v_{i+1}$ as a test function in \eqref{stwsi}, and using the fact that 
$$\left( u_{i+1}-u_i \right) v_{i+1} \geq 0$$
and $u_{i+1}\nabla v_{i+1} = \nabla v_{i+1}$, we get 
\begin{equation} 
	\eps\: \int_\Omega  \left|\nabla v_{i+1}  \right|^p   \leq  	\eps\: \int_\Omega   \left(  f_i\: v_{i+1}  +    V  \cdot \nabla v_{i+1}\right)  ,
\end{equation}		
which  implies that  
\begin{equation} 
	\iint_Q \left|\nabla v_\eps \right|^p   \leq   \iint_Q\left(  f_\eps\: v_{\eps}  +    V  \cdot \nabla v_{\eps}\right) ,
\end{equation}
where 
$$f_\eps =\sum_{i=0}^{n-1} f_{i} \: \chi_{[t_i,t_{i+1})},\quad \hbox{in }\Omega.$$  

Reasoning as in the proof of \eqref{hestimate0}, we get  	
$$\int_{\Omega} \left| \nabla v_\eps \right|^p dx  \leq  C (N,p,\Omega) \left( \int_{\Omega} \left| V  \right|^{p^\prime} dx  + \int_{\Omega} \left| f_\eps  \right|^{p^\prime} dx \right).$$
This implies that $0\leq v_\eps$ is bounded in $L^p \left( 0,T;W^{1,p}_D(\Omega) \right)$ and that there exists $0\leq v\in L^p \left( 0,T;W^{1,p}_D(\Omega) \right)$ such that, taking a subsequence if necessary, 
$$v_\eps \longrightarrow v,\quad \hbox{in } L^p \left( 0,T;W^{1,p}_D(\Omega) \right)-\hbox{weak}.$$
Combining this with \eqref{convueps} and the fact that $u_\eps\in \sign^+ (\eps)$, we deduce moreover that $u \in \sign^+ (v)$, a.e. in $Q$.  Now, as usual in nonlinear semigroup theory for evolution problems, we consider   
$$\tilde u_\eps = \sum_{i=0}^{n-1} \frac{(t-t_i)u_{i+1} - (t-t_{i+1}) u_i}{\eps } \: \chi_{[t_i,t_{i+1})},$$
which converges to $u$ as well in $\C \left( [0,T);L^1(\Omega) \right)$. For any test function $\xi \in W^{1,p}_D(\Omega)$, we have
$$ \frac{d}{dt} \int_\Omega \tilde u_\eps \: \xi \, dx + \int_\Omega \left( \nabla^{p-1} v_\eps -  u_\eps \:V \right) \cdot  \nabla \xi \, dx = \int_\Omega f_\eps \: \xi \, dx, \quad \hbox{in }{\D}'([0,T)).$$
So, letting $\eps\to 0$ and using the convergence of $(\tilde u_\eps,u_\eps,v_\eps,f_\eps)$ to $(u,u,v,f)$, we deduce that $(u,v)$ is a weak solution of \eqref{cmef}, and $v$ satisfies \eqref{newestimate}.
\end{proof}
	
	\medskip 
	\begin{proofth}{Proof of Theorem \ref{existcmef}} The proof follows directly from Lemma \ref{lmildweak}
		
	\end{proofth} 
	
	\medskip 
	
	\begin{proofth}{Proof of Theorem \ref{Theorem1}} The existence of a weak solution is directly established by Theorem \ref{existcmef}. Uniqueness is ensured by   Theorem \ref{compcmef} and Corollary \ref{Cuniq}.
		
	\end{proofth} 
	
	\section{Asymptotic behaviour as $p\to\infty$}\label{Slimitp}
	
	This section contains the proof of Theorem \ref{Theorem2}, namely the study of the limit as $p \to \infty$ of the solution of \eqref{cmef}. 
	We start with appropriate \textit{a priori} estimates independent of $p$. We obviously have 
	\begin{equation}
		\label{linfinitybound}
		\left\| u_p\right\|_{L^\infty (Q)} \leq 1.
	\end{equation} 
Thanks to \eqref{newestimate} and \eqref{HypV0}, we have
	\begin{eqnarray*}
		\iint_Q \left|\nabla v_p \right|^p & \leq & \iint_Q \left( f - \nabla \cdot V \right) v_p\\
		& \leq & \frac{1}{p^\prime \epsilon^{p^\prime}} \iint_Q \left| f - \nabla \cdot V \right|^{p^\prime} + \frac{\epsilon^p}{p} \iint_Q \left| v_p \right|^p\\
		& \leq & \frac{1}{p^\prime \epsilon^{p^\prime}} \iint_Q \left| f - \nabla \cdot V \right|^{p^\prime} + \frac{C_p^p\epsilon^p}{p} \iint_Q \left| \nabla v_p \right|^p,
	\end{eqnarray*}
	using Young's and Poincar\'e's inequalities.
	We now take
	$$\epsilon = \left( \frac{p}{2}\right)^{1/p} \frac{1}{C_p}$$
	to get
	$$\iint_Q \left|\nabla v_p \right|^p  \leq (p-1) \left(\frac{2C_p}{p} \right)^{p^\prime} \iint_Q \left| f - \nabla \cdot V \right|^{p^\prime}.$$
	Now, for any $q \geq 1$ and $p \geq q$, we have, by H\"older's inequality,
	\begin{eqnarray*}\iint_Q \left|\nabla v_p \right|^q  & \leq & \left| Q \right|^{1-\frac{q}{p}} \left( \iint_Q \left|\nabla v_p \right|^p \right)^{\frac{q}{p}}\\
		& \leq & \left| Q \right|^{1-\frac{q}{p}}  (p-1)^{\frac{q}{p}} \left(\frac{2C_p}{p} \right)^{^{\frac{q}{p-1}}} \left( \iint_Q \left| f - \nabla \cdot V \right|^{p^\prime} \right)^{\frac{q}{p}}.
	\end{eqnarray*}
	Taking the limit as $p \to \infty$, we obtain
	\begin{equation}
		\label{w1qbound}
		\lim_{p\to \infty} \left\| \nabla v_p \right\|_q^q \leq \left| Q \right|,
	\end{equation}
	since $C_p \to C$ (see \cite[page 110]{Ziemer}) and
	$$\left\| f - \nabla \cdot V \right\|_{p^\prime}^{\frac{q}{p-1}} \longrightarrow \left\| f - \nabla \cdot V \right\|_{1}^{0} = 1.$$
	Using again Poincar\'e's inequality, we conclude that $(v_p)_p$ is bounded in $L^q \left(0,T;W^{1,q}(\Omega) \right)$. 
	
	\medskip
	
	\begin{proofth}{Proof of Theorem \ref{Theorem2}}
		From \eqref{linfinitybound} and \eqref{w1qbound}, we find a pair 
		$$ \left( u,v\right) \in L^\infty (Q) \times L^q \left(0,T;W^{1,q}(\Omega) \right)$$
		such that, for subsequences (that we relabel for convenience),
		$$\begin{array}{cccrl}
			u_p & \rightharpoonup & u & \mathrm{in} & L^\infty (Q)-\mbox{weak-}\ast;\\
			\\
			v_p & \rightharpoonup & v & \mathrm{in} & L^q \left(0,T;W^{1,q}(\Omega)\right)-\mathrm{weak}.
		\end{array}$$
		Moreover, we have $0\leq v ,$ 
		 $0 \leq u \leq 1$ 
		and
		$$\left\| \nabla v \right\|_q \leq \left| Q \right|^{1/q}, \qquad \forall q>1.$$
		So, taking $q \to \infty$, we get
		$$\left\| \nabla v \right\|_{\infty} \leq 1.$$
		
		Next, we show that $\partial_t u_p$ is uniformly weakly bounded. Indeed, for $\varphi \in \D (Q)$,
		\begin{eqnarray*}
			\iint_Q u_p \partial_t  \varphi & = &  - \iint_Q \left|\nabla v_p \right|^{p-2} \nabla v_p \cdot \nabla \varphi +\iint_Q  u_p  \: V \cdot \nabla \varphi + \iint_Q f \varphi\\
			& \leq & \| \nabla \varphi \|_{\infty} \left\{ \iint_Q \left|\nabla v_p \right|^{p-1} + \iint_Q  |V| + C \| f \|_{1} \right\}\\
			& \leq & C\: \| \nabla \varphi \|_{\infty}.
		\end{eqnarray*}
		We can then apply \cite[Proposition 1.4]{ACM_17}, to conclude that 
		$$u \in \sign^+ (v).$$

	Thanks to \eqref{newestimate}, one sees that 
	\begin{equation}
\begin{array}{c}
		- \delta \frac{d}{dt} \int_{\Omega} u \xi + 	  \iint_Q \left|\nabla v_p \right|^{p-2} \nabla v_p \cdot \nabla (v_p - \delta \xi)   \\  \leq   \iint_Q  u_p  \: V \cdot \nabla \left( v_p -\delta \xi \right) + \iint_Q f \left( v_p - \delta\xi \right),  \quad \hbox{ in }\D'([0,T)),
		\end{array} 
 	\end{equation}
for any $\xi \in W_D^{1,\infty} (\Omega)$ and $0<\delta <1$, where we use again that 
$$\int_\Omega  u_p  \: V \cdot \nabla  v_p  = \int_\Omega    V \cdot \nabla  v_p, \quad \hbox{ a.e } t\in [0,T).$$
Our aim now is to let $p\to \infty$  :    
	\begin{enumerate}
	 \item for the first term, we have
	\begin{equation}
  - \delta\:  \lim_{p\to\infty }    \frac{d}{dt} \int_{\Omega} u_p \xi  \quad \rightharpoonup  \quad  - \delta\:  \frac{d}{dt} \int_{\Omega} u \xi,  \quad \hbox{ in }\D'([0,T)) ; 
	\end{equation}

	\medskip
	
	\item concerning the second term, by monotonicity and \eqref{w1qbound}, we have
	$$
	\hspace*{-1cm} \liminf_{p\to\infty }\int_\Omega \left|\nabla v_p \right|^{p-2} \nabla v_p \cdot \nabla (v_p - \delta\: \xi) 
	$$
	$$
	\hspace*{1.5cm} \geq \lim_{p\to\infty} \delta^{p-1}\: \int_\Omega \left|\nabla \xi \right|^{p-2} \nabla \xi \cdot \nabla (v_p - \delta\: \xi)= 0;\quad \hbox{ in }\D'([0,T))  
	$$
	
	\medskip
	
	\item as for the first term on the right-hand side, we get
	
	\begin{eqnarray*}
		\lim_{p\to\infty } \int_\Omega  u_p  \: V \cdot \nabla \left( v_p - \delta\: \xi \right) & = & \lim_{p\to\infty } \int_\Omega  u_p  \: V \cdot \nabla v_p  \\
		& & -\delta\:   \lim_{p\to\infty }    \int_\Omega  u_p  \: V \cdot \nabla \xi \\
		& = & \lim_{p\to\infty } \int_\Omega  V \cdot \nabla v_p  \\
		& & -\delta\:  \lim_{p\to\infty }   \int_\Omega  u_p  \: V \cdot \nabla \xi \\
		& = &  \int_\Omega  u  \: V \cdot \nabla (v- \delta\:  \xi), \quad \hbox{ in }\D'([0,T)) 
	\end{eqnarray*}
	since $u_p\in \sign^+ (v_p)$ and $u\in \sign^+ (v)$, a.e. in $Q;$ 
	
	\medskip
	
	\item for the second term on the right-hand side, the passage to the limit is straightforward. 
	
\end{enumerate}
		
		\medskip
		
		So, for any $0<\delta <1,$ we get 
		$$	 	
	 - \delta\:  \frac{d}{dt} \int_{\Omega} u \xi   - \int_{\Omega}   u  \: V \cdot \nabla \left( v -\delta \xi \right)  \leq \int_{\Omega}   f \left( v - \delta\xi \right) \quad \hbox{ in }\D'([0,T)) .
		$$			
		Letting at last $\delta \to 1$, we obtain the result. 
		
	\end{proofth}
	
	\appendix
	
	\section{}
	
	We need two lemmas to prepare for the proof of Proposition \ref{prenormal}.
	
	\begin{aplemma}\label{leqp+}  	 
		If $(u,v)$ is a weak solution of \eqref{cmef}, then
		\begin{equation}\label{eqp+1}
			-\Delta_p v + \left( \nabla \cdot V-f \right) \sign_0(v) \leq 0 \quad \hbox{in } \D^\prime \left( (0,T)\times \overline \Omega \right).
		\end{equation}
	\end{aplemma}
	
	\begin{proof}
		
		We extend $v$ to $\R \times\Omega$ by $0$, for any $t\not\in (0,T)$, and, for any $h>0$, we consider, with $\xi \in \D(\overline \Omega)$ and $\psi \in \D \left( (0,T) \right)$,
		$$ 
		\Phi^h (t,x)=\xi(x)\, \psi(t) \, \frac{1}{h} \int_t^{t+h} \heps (v(s,x))\, ds, \quad \hbox{for a.e. }(t,x)\in Q,
		$$
		where $\psi$ is extended, in turn, to $\R$ by $0$, and $\heps$ is given in $\R$ by
		$$
		\heps(r)=\left\{ 
		\begin{array}{ll}   
			1 \quad & \hbox{ if } r>\eps \\   
			r/\eps &  \hbox{ if } \left|r\right|\leq \eps\\ 
			-1 & \hbox{ if } r< - \eps,
		\end{array}\right.
		$$
		for $\eps >0.$  It is clear that $\Phi^h \in W^{1,p} \left( 0,T; W^{1,p}_D(\Omega) \right) \cap L^\infty(Q)$ is an admissible test function for the weak formulation, so that 
		\begin{equation}\label{evolh0}
			- \iint_Q u \, \partial_t \Phi^h + \iint_Q \left( \nabla^{p-1} v - V \,  u  \right) \cdot \nabla \Phi^h  = \iint_Q f\, \Phi^h.
		\end{equation}
		Observe that 
		\begin{eqnarray}
			\iint_Q  u  \, \partial_t \Phi^h & = & \iint_Q  u \, \xi \, \partial_t \psi \, \frac{1}{h} \int_t^{t+h} \heps(v((s)) \, ds \nonumber\\
			& & + \iint_Q u(t) \, \frac{\heps (v (t+h)) - \heps (v (t))}{h} \, \psi(t)\,  \xi. \label{rel0} 
		\end{eqnarray}
		Moreover, using the fact that, for a.e. $t\in (0,T)$, $0\leq  u(t) \leq 1$,  ${\heps}\geq 0$ and $ \heps(0)=0$,  we have
		$$u(t,x) \, \heps(v(t,x)) = \heps(v(t,x))$$ 
		and 	
		$$  u(t,x)\: \heps(v(t+h,x)) \leq \heps(v(t+h,x)), \quad \mbox{a.e. } (t,x) \in Q.$$  
		So, for $h>0$ small enough, we have 
		\begin{eqnarray*}
			\iint_Q u(t) \, \frac{\heps(v (t+h)) - \heps(v (t))}{h} \, \psi(t)\,  \xi \\  	
			\leq \iint_Q \frac{\heps(v (t+h)) - \heps(v (t))}{h} \, \psi(t) \, \xi  \\	
			\leq \iint_Q \frac{\psi (t-h) - \psi (t)}{h} \, \heps(v (t)) \, \xi.
		\end{eqnarray*}
		This implies that
		$$
		\limsup_{h\to 0} \iint_Q u(t) \, \frac{\heps(v (t+h)) - \heps(v (t))}{h} \, \psi(t) \, \xi \leq - \iint_Q \partial_t \psi \,     \heps(v(t)) \, \xi,
		$$
		so that, by letting $h\to 0$ in \eqref{rel0}, we get
		$$	
		\lim_{h\to 0} \iint_Q u  \, \partial_t \Phi^h \leq 0.
		$$ 
		Then, by letting $h\to 0$ in \eqref{evolh0}, we obtain 
		\begin{equation}\label{ajout0}
			\iint_Q \left( \nabla^{p-1} v - V\:  u  \right) \cdot \nabla \left( \heps(v (t)) \, \xi \right) \psi  \leq \iint_Q f \, \heps(v(t)) \, \xi \, \psi. 	
		\end{equation}
		On the other hand, using again the fact that $u \, \heps(v)= \heps(v)$, a.e. in $Q$, we have
		\begin{eqnarray*}
			\iint_Q \left( \nabla^{p-1} v - V\:  u  \right) \cdot \nabla \left( \heps(v (t)) \, \xi \right) \psi \hspace*{4cm}\\ 
			= \iint_Q \heps(v) \, \nabla^{p-1} v \cdot \nabla \xi \, \psi + \iint_Q  \left|  \nabla v \right|^p ( {\heps})'(v)\, \xi \, \psi \\
			- \iint_Q  V \cdot \nabla \left( \xi \heps(v) \right) \psi \\  
			\geq \iint_Q \heps(v) \, \nabla^{p-1} v \cdot \nabla \xi \, \psi + \iint_Q \nabla \cdot V  \left( \xi \heps(v) \right) \psi, 
		\end{eqnarray*}
		using \eqref{HypV0} and the fact that $\left| \nabla v\right|^p ({\heps})'(v) \geq 0$.
		Thanks to \eqref{ajout0}, this implies  
		$$
		\iint_Q \nabla^{p-1} v \cdot \nabla \xi \, {\heps}(v)\, \psi + \iint_Q \nabla \cdot  V \,  \heps(v) \, \xi \, \psi \leq \iint_Q f \, \heps(v) \, \xi \, \psi. 
		$$ 
		Letting $\eps\to 0$, we obtain \eqref{eqp+1}.      
		
	\end{proof}
	
	We now state and prove the second lemma.
	
	\begin{aplemma} \label{lrenormal}
		Let $u\in L^1_{loc}(Q)$,  $F\in L^1_{loc}(Q)^N$ and $J_1 \in L^1_{loc}(Q)$ be such that
		\begin{equation}\label{form1}
			\partial_t u + V\cdot \nabla u -\nabla \cdot F = J_{1} \quad \hbox{in } \mathcal D^\prime(Q),
		\end{equation}
		where $V\cdot \nabla u$ is taken in the sense $V \cdot \nabla u = \nabla \cdot (u \, V) - u \, \nabla \cdot V$ in $\D^\prime(Q)$. If
		\begin{equation}\label{form2}
			-\nabla \cdot F \leq J_{2} \quad \hbox{in } \D^\prime(Q),
		\end{equation}
		for some $J_2 \in L^1_{loc}(Q)$, then
		\begin{equation} \label{techineq}
			\partial_t \beta (u) + V\cdot \nabla \beta(u) - \nabla \cdot F \leq {J_{1}} \beta^\prime(u) +  {J_{2}}(1-\beta^\prime(u))         \quad \hbox{in }  \D^\prime(Q),
		\end{equation}
		for any $\beta \in \mathcal C^1(\R)$ such that $\beta^\prime \leq 1$.
	\end{aplemma}
	
	\begin{proof}    
		
		We set 
		$$Q_\eps := \left\{ (t,x)\in Q\: : \: d((t,x),\partial Q) > \eps \right\}.$$  
		Moreover, for any $z\in L^1_{loc}(Q)$, we denote by $z_\eps$ the usual regularization of $z$ by convolution given by 
		$$z_\eps := z \star \rho_\eps, \quad \hbox{in }Q_\eps, $$ 
		where $\rho_\eps$ is the standard mollifying sequence in $\R \times \R^N$.  
		We can show that \eqref{form1} and \eqref{form2} imply, respectively,
		\begin{equation}\label{formeps1}
			\partial_t u_\eps + V \cdot \nabla u_\eps -\nabla \cdot F_\eps = {J_{1}}_\eps + \C_\eps \quad \hbox{in }  Q_\eps
		\end{equation}
		and 
		\begin{equation}\label{formeps2}
			-\nabla \cdot F_\eps \leq {J_{2}}_\eps \quad \hbox{in }  Q_\eps, 
		\end{equation}
		where $\C_\eps$ is the usual commutator given by
		$$ \C_\eps:= V\cdot \nabla u_\eps - (V\cdot \nabla u)_\eps.$$  
		Here $(V\cdot \nabla u)_\eps$ needs to be understood in the sense 
		$$
		(V\cdot \nabla u)_\eps = (u\, V)\star  \nabla \rho_\eps - (u\, \nabla \cdot V)\star \rho_\eps, \quad \hbox{in }Q_\eps.
		$$
		Multiplying \eqref{formeps1} by $\beta^\prime(u_\eps)$ and \eqref{formeps2} by $1-\beta^\prime(u_\eps)$, and adding the resulting equations, we obtain
		$$\beta^\prime (u_\eps) \, \partial_t u_\eps + \beta^\prime(u_\eps) \, V\cdot \nabla u_\eps -\nabla \cdot F_\eps \leq \C_\eps \,  \beta^\prime(u_\eps) + {J_{1}}_\eps \beta^\prime(u_\eps) + {J_{2}}_\eps(1-\beta^\prime(u_\eps) )$$
		and
		\begin{equation}\label{formeps}
			\partial_t \beta(u_\eps) +	V\cdot \nabla \beta(u_\eps) -\nabla \cdot F_\eps \leq \C_\eps \, \beta^\prime(u_\eps)  + {J_{1}}_\eps\beta^\prime(u_\eps) + {J_{2}}_\eps(1-\beta^\prime(u_\eps)),
		\end{equation}
		in $Q_\eps$. Since $V \in W^{1,1}_{loc} (\Omega)$ and $\nabla \cdot V \in L^\infty(\Omega)$, it is well-known that taking a subsequence if necessary, the commutator converges to $0$ in $L^1_{loc}(Q)$, as $\eps \to 0$ (see, for instance, \cite{Ambrosio}). Thus, letting $\eps\to 0$ in \eqref{formeps}, we obtain \eqref{techineq}.
		
	\end{proof}
	
	We are now ready for the proof of Proposition \ref{prenormal}.
	
	\medskip
	
	\begin{proofth}{Proof of Proposition \ref{prenormal}}
		Due to Lemma \ref{leqp+}, and using the fact that
		$$ 	
		\nabla \cdot V \, \sign_0(v)= u \, \nabla \cdot V \, \sign_0(v),
		$$ 
		we see that \eqref{form1} and \eqref{form2} are fulfilled with
		$$ 
		F:=  \nabla^{p-1} v,  \qquad  J_{1}:= f- u \, \nabla \cdot V
		$$
		and  	
		$$ 
		J_{2}:= \left( f-u \, \nabla \cdot V \right) \sign_0(v). 
		$$
		Applying Lemma \ref{lrenormal}, for any $\beta \in \mathcal C^1(\R)$ such that $\beta^\prime \leq 1$, we deduce that
		$$
		\partial_t \beta (u) - \Delta_p v + 	V\cdot \nabla \beta(u) + \left( u\, \nabla \cdot V-f \right)  \beta^\prime(u) \hspace*{2cm}
		$$
		$$
		+ \left( f- u \nabla \cdot V  \right) \sign_0(v) ( \beta'(u)-1)  \leq 0 \quad \hbox{in }  \D^\prime (Q).
		$$
		Using again the fact that $\nabla \cdot V \, \sign_0(v)	=u\, \nabla \cdot V \, \sign_0(v)$ this implies that
		$$
		\partial_t \beta (u) - \Delta_p v  + 	V \cdot \nabla \beta(u) +  u\, \nabla \cdot V \, \beta^\prime(u) + u\nabla \cdot V \, \sign_0(v) (1-\beta^\prime(u))
		$$
		$$     
		\leq f  \left( \sign_0(v) (1-\beta'(u) ) + \beta^\prime(u) \right) \quad \hbox{in } \D^\prime(Q),$$
		and then 
		$$
		\partial_t \beta (u) - \Delta_p v + V \cdot \nabla \beta(u) +  u \, \nabla \cdot V \left( \beta^\prime(u) \chi_{[v=0]} +  \sign_0(v)\right) 
		$$    
		$$
		\leq f  \left( \beta^\prime(u) \chi_{[v=0]} + \sign_0(v)  \right)  \quad \hbox{in }  \D^\prime(Q).
		$$		
	\end{proofth}
	
\medskip

{\small \noindent{\bf Acknowledgments.} 
This publication is based upon work supported by King Abdullah University of Science and Technology (KAUST) under Award No. ORFS-CRG12-2024-6430. JMU is partially supported by UID/00324 - Centre for Mathematics of the University of Coimbra. This work was initiated during a visit of NI to KAUST, and he would like to express his gratitude to the institution for the hospitality.}

\end{document}